\documentclass[12pt]{amsart}
\usepackage{graphicx}
\usepackage[headings]{fullpage}
\usepackage{amssymb,epic,eepic,epsfig,amsbsy,amsmath,amscd}
\numberwithin{equation}{section}
                        \theoremstyle{plain}
\usepackage{mathrsfs}

\newcommand\no[1]{}

\newtheorem{theorem}{Theorem}[section]
\newtheorem{thm}{Theorem}
\newtheorem{lemma}[theorem]{Lemma}

\newtheorem{proposition}[theorem]{Proposition}
\newtheorem{conjecture}{Conjecture}
\newtheorem{cor}{Corollary}

\theoremstyle{definition}

\def\BC{\mathbb C}

\def\BR{\mathbb R}

\def\fb{\mathfrak b}

\def\la{\langle}
\def\ra{\rangle}

\DeclareMathOperator{\tr}{\mathrm tr}

\def\ve{\varepsilon}
\def\be { \begin{equation} }
\def\ee { \end{equation} }

\begin{document}

\title[Nonabelian representations and signatures]{Nonabelian representations and signatures of double twist knots}

\author[Anh T. Tran]{Anh T. Tran}
\address{Department of Mathematical Sciences, University of Texas at Dallas, Richardson, TX 75080, USA}
\email{att140830@utdallas.edu}

\dedicatory{Dedicated to Professor Jozef Przytycki on the occasion of his 60th birthday}

\begin{abstract}
A conjecture of Riley about the relationship between real parabolic representations and signatures of two-bridge knots is verified for double twist knots.
\end{abstract}

\thanks{2010 {\em Mathematics Classification:} Primary 57N10. Secondary 57M25.}

\maketitle

\section{Introduction}

The motivation of this note is a conjecture of Riley about the relationship between real parabolic representations and signatures of two-bridge knots. Two-bridge knots are those knots admitting a projection with only two mixama and two minima. The double branched cover of $S^3$ along a two-bridge knot is a lens space $L(p,q)$, which is obtained by doing a $p/q$ surgery on the unknot. Such a two-bridge knot is denoted by $\fb(p,q)$. Here $p$ and $q$ are relatively prime integers, and one can always assume that $p > q \ge 1$. It is known that $\fb(p',q')$ is ambient isotopic to $\fb(p,q)$ if and only if $p'=p$ and $q' \equiv q^{\pm 1} \pmod{p}$, see e.g. \cite{BZ}. The knot group of the two-bridge knot $\fb(p,q)$ has a presentation of the form $\la a, b \mid wa = b w \ra$ where $a,b$ are meridians, $w=a^{\ve_1} b^{\ve_2} \cdots a^{\ve_{p-2}}b^{\ve_{p-1}}$ and $\ve_j=(-1)^{\lfloor jq/p \rfloor}$. We will call this presentation the Schubert presentation of the knot group of $\fb(p,q)$. 

We now study representations of knot groups into $SL_2(\BC)$. A representation is called nonabelian if its image is a nonabelian subgroup of $SL_2(\BC)$. Suppose $\rho: \pi_1(\fb(p,q)) \to SL_2(\BC)$ is a nonabelian representation. Up to conjugation, one can assume that $$\rho(a) = \left[ \begin{array}{cc}
s & 1 \\
0 & s^{-1} \end{array} \right] \quad \text{and} \quad \rho(b) = \left[ \begin{array}{cc}
s & 0 \\
2-y & s^{-1} \end{array} \right]$$
where $s \not= 0$ and $y \not= 2$ satisfy the matrix equation $\rho(wa)=\rho(bw)$. Riley \cite{Ri} showed that this matrix equation is actually equivalent to a single polynomial equation $\widetilde{\Phi}_{\fb(p,q)}(s,y)=0$. The polynomial $\widetilde{\Phi}_{\fb(p,q)}(s,y)$ is called the Riley polynomial of $\fb(p,q)$. It can be chosen so that $\widetilde{\Phi}_{\fb(p,q)}(s,y)=\widetilde{\Phi}_{\fb(p,q)}(s^{-1},y)$, and hence it can be considered as a polynomial in the variables $x=s+s^{-1}$ and $y$.  Note that $x=\tr\rho(a)=\tr\rho(b)$ and $y=\tr\rho(ab^{-1})$. With these new variables, we write $\Phi_{\fb(p,q)}(x,y)$ for the Riley polynomial of $\fb(p,q)$. See \cite{Le} for an alternative description of $\Phi_{\fb(p,q)}(x,y)$ in terms of traces.

A nonabelian representation of a knot group into $SL_2(\BC)$ is called parabolic if the trace of a meridian is equal to $2$. The zero set in $\BC$ of the polynomial $\Phi_K(2,y)$ of a two-bridge knot $K$ describes the set of all parabolic representations of the knot group of $K$ into $SL_2(\BC)$, see \cite{Ri-parabolic}. Riley made the following conjecture about the relationship between the number of real roots of $\Phi_K(2,y)$ and the signature $\sigma(K)$ of $K$.

\begin{conjecture} \cite{Ri}
Suppose $K$ is a two-bridge knot. Then the equation $\Phi_K(2,y)=0$ has at least $\frac{1}{2}|\sigma(K)|$ real solutions.
\end{conjecture}

In this note we will verify Conjecture 1 for a large class of two-bridge knots, called double twist knots. Let $J(k,l)$ be the two-bridge knot/link in Figure 1, where $k,l \not= 0$ denote 
the numbers of half twists in the boxes. Positive (resp. negative) numbers correspond 
to right-handed (resp. left-handed) twists. 
Note that $J(k,l)$ is a knot if and only if $kl$ is even. The knot $J(2,2n)$ is known as a twist knot. In general, the knot/link $J(k,l)$ is called a double twist knot/link. It is known that 
\begin{eqnarray*}
J(2m,2n) &=& \fb(4mn-1,4mn-2n-1),\\
J(2m,-2n) &=& \fb(4mn+1,4mn-2n+1),\\
J(2m+1,2n) &=& \fb(4mn+2n-1,4mn-1),\\
J(2m+1,-2n) &=& \fb(4mn+2n+1,4mn+1)
\end{eqnarray*}
for positive integers $m$ and $n$. Moreover, their signatures are $2$, $0$, $2-2n$ and $2n$ respectively, see e.g. \cite{Mu}. For more information about $J(k,l)$, see \cite{HS, MPL}.

\begin{figure}[th]
\centerline{\psfig{file=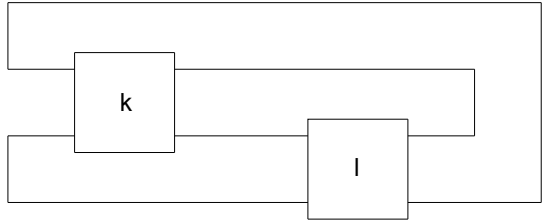,width=3.5in}}
\vspace*{8pt}
\caption{The two-bridge knot/link $J(k,l)$. }
\end{figure} 

The main results of this note are as follows.

\begin{thm} \label{main1}
Suppose $m$ and $n$ are positive integers. If $x_0 \in \BR$ satisfies $\sqrt{4-(mn)^{-1}}< |x_0| \le 2$, then 

(i) the equation $\Phi_{J(2m,2n)}(x_0,y)=0$ has exactly one real solution $y$, and

(ii) the equation $\Phi_{J(2m,-2n)}(x_0,y)=0$ has no real solutions $y$,

\end{thm}

\begin{thm} \label{main2}
Suppose $m$ and $n$ are positive integers. If $x_0 \in \BR$ satisfies $|x_0| \ge 2\cos \frac{\pi}{4m+2}$, then

(i) the equation $\Phi_{J(2m+1,2n)}(x_0,y)=0$ has at least $n-1$ real solutions $y$, and

(ii) the equation $\Phi_{J(2m+1,-2n)}(x_0,y)=0$ has at least $n$ real solutions $y$.
\end{thm}

As a consequence of Theorems \ref{main1} and \ref{main2} we have the following.

\begin{cor} 
Conjecture 1 holds true for double twist knots.
\end{cor}

Results in Theorems \ref{main1} and \ref{main2} are related to the problem of determining the existence of real parabolic representations in the study of left-orderability of the fundamental groups of cyclic branched covers of two-bridge knots, see \cite{Hu, Tr}.  A group $G$ is called left-orderable if there exists a strict total ordering $<$ on its elements such that $g<h$ implies $fg<fh$ for all elements $f,g,h \in G$. Left-orderable groups have
 attracted much attention because of its connection
to L-spaces, a class of rational homology 3-spheres defined by Ozsvath and Szabo \cite{OS} using
Heegaard Floer homology, via a conjecture of Boyer, Gordon and Watson \cite{BGW}. This
conjecture states that an irreducible rational homology 3-sphere is an L-space if and
only if its fundamental group is not left-orderable.

In \cite{BGW, Hu} a sufficient condition for the fundamental
group of the $n^{th}$ cyclic branched cover of $S^3$ along a prime knot to be left-orderable is given, in terms of representations of the knot group. In particular, from the proof of \cite[Theorem 4.3]{Hu} it follows that if the knot group of a two-bridge knot $K$ has a real parabolic representation then the $n^{th}$ cyclic branched cover of $S^3$ along $K$ has left-orderable fundamental group for sufficiently large $n$. Consequently, if Conjecture 1 holds true then we would have the following: {\em if $K$ is a two-bridge knot with nonzero signature then the $n^{th}$ cyclic branched cover of $S^3$ along $K$ has left-orderable fundamental group for sufficiently large $n$.} 

The rest of this note is devoted to the proofs of Theorems \ref{main1} and \ref{main2}. In each proof we begin by computing the knot group and then compute the Riley polynomial $\Phi_K(x,y)$ of a double twist knot $K$. Finally we study real roots of the polynomial $\Phi_K(2,y)$.

\section{Proof of Theorem \ref{main1}}

\subsection{$J(2m,2n)$} The knot $J(2m,2n)$ is the two-bridge knot $\fb(4mn-1,4mn-2n-1)$.

\subsubsection{The knot group and the Riley polynomial}

\begin{lemma} \label{w1}
Suppose $1 \le j \le 4mn-2$. Write $j=2mq+r$, where $0 \le r \le 2m-1$. Then 
$\ve_j = (-1)^{q+r-1}.$
\end{lemma}

\begin{proof} We have
$$
\frac{(4mn-2n-1)j}{4mn-1} = 2mq+r-\frac{2n(2mq+r)}{4mn-1} =(2m-1)q+r-\frac{q+2nr}{4mn-1}.
$$
Since $ 0 < \frac{q+2nr}{4mn-1}<1$, we obtain $\lfloor \frac{(4mn-2n-1)j}{4mn-1} \rfloor= (2m-1)q+r-1$. Hence $\ve_j=(-1)^{q+r-1}$.
\end{proof}

Lemma \ref{w1} implies that the Schubert presentation of the knot group of $J(2m,2n)$ is $\pi_1(J(2m,2n)) = \la a,b \mid wa=bw \ra$ where $$w = a(b^{-1}a)^{m-1}\big[ (ba^{-1})^{m}(b^{-1}a)^{m} \big]^{n-1}(ba^{-1})^{m-1}b.$$

The Riley polynomial is computed from the matrix equation $\rho(wa)=\rho(bw)$. The proof of the following is similar that of \cite[Proposition 2.5]{MT} and hence is omitted.

\begin{proposition}
One has $\Phi_{J(2m,2n)}(x,y)= S_n(t) -  \mu S_{n-1}(t)$ where 
\begin{eqnarray*}
t &=& 2+(y-2)(y+2-x^2)S^2_{m-1}(y),\\
\mu &=& 1 + (y+2-x^2)S_{m-1}(y)(S_{m}(y)-S_{m-1}(y)).
\end{eqnarray*}
\end{proposition}

Here $S_k(z)$'s are the Chebychev polynomials of the second kind defined by $S_0(z)=1$, $S_1(z)=z$ and $S_{k}(z)=zS_{k-1}(z)-S_{k-2}(z)$ for all integers $k$. Note that $S_k(2)=k+1$ and $S_k(-2)=(-1)^k (k+1)$. Moreover if $z=\lambda+\lambda^{-1}$, where $\lambda \not= \pm 1$, then $S_k(z)=\frac{\lambda^{k+1}-\lambda^{-(k+1)}}{\lambda-\lambda^{-1}}$.  

In the proofs of Theorems \ref{main1} and \ref{main2} below, we will need the following properties of $S_k(z)$'s whose proofs can be found, for example, in \cite{Tr} and references therein.

\smallskip

(1) $S^2_k(z) + S^2_{k-1}(z) - z S_k(z)S_{k-1}(z)=1$.

\smallskip

(2) If $z$ is a real number between $-2$ and $2$, then $|S_{k-1}(z)| \le |k|$.

\smallskip

(3) $S_k(z) = \prod_{j=1}^{k} (z-2\cos \frac{j\pi}{k+1})$ and $S_k(z) - S_{k-1}(z) =  \prod_{j=1}^{k} (z-2\cos \frac{(2j-1)\pi}{2k+1})$ for $k \ge 1$.

\subsubsection{Proof of Theorem \ref{main1}(i)} Suppose $x_0 \in \BR$ satisfies $\sqrt{4-(mn)^{-1}}< |x_0| \le 2$. We want to prove that the equation $\Phi_{J(2m,2n)}(x_0,y)=0$ has exactly one real solution $y$. 

The proof consists of two steps.\\
\textit{Step 1: The equation $\Phi_{J(2m,2n)}(x_0,y)=0$ has no real solutions $y \le 2$.}

The proof of Step 1 is similar to that of \cite[Proposition 2.7]{Tr}. Suppose $y \le 2$ and $\Phi_{J(2m,2n)}(x_0,y)=0$. Then $S_n(t) =  \mu S_{n-1}(t)$. Since $S^2_{n}(t)+S^2_{n-1}(t)-tS_{n}(t)S_{n-1}(t)=1$ we have $(\mu^2+1-\mu t)S^2_{n-1}(t)=1$. 

By a direct calculation we have $\mu^2+1-\mu t = (y+2-x_0^2)S^2_{m-1}(y)(t+2-x_0^2)$. Hence
\begin{equation} \label{yt1}
(y+2-x_0^2)S^2_{m-1}(y)(t+2-x_0^2)S^2_{n-1}(t)=1.
\end{equation}

Since $(y-2)(y+2-x_0^2)S^2_{m-1}(y)=t-2$, equation \eqref{yt1} implies that $$(t-2)(t+2-x_0^2)S^2_{n-1}(t)=y-2 \le 0.$$ 
Hence $x_0^2-2<t \le 2$. Then $(y-2)(y+2-x_0^2)=(t-2)/S^2_{m-1}(y) \le 0$, which implies that $x_0^2-2< y \le 2$. Since $|y| \le 2$, we have $|S_{m-1}(y)| \le m$. Similarly $|S_{n-1}(t)| \le n$. Hence
$$(y+2-x_0^2)S^2_{m-1}(y)(t+2-x_0^2)S^2_{n-1}(t) \le (4-x_0^2)^2m^2n^2<1$$
which contradicts equation \eqref{yt1}.\\
\textit{Step 2: The equation $\Phi_{J(2m,2n)}(x_0,y)=0$ has exactly one real solution $y > 2$.}

Consider the following real functions on $(2,\infty)$:
\begin{eqnarray*}
f(y) &=& \Phi_{J(2m,2n)}(x_0,y) =  S_n(t) - \mu S_{n-1}(t),\\
g(y) &=& (y+2-x_0^2)S^2_{m-1}(y)(t+2-x_0^2)S^2_{n-1}(t)-1.
\end{eqnarray*}
Note that $g(y)$ comes from the equation \eqref{yt1}. In Step 1 we showed that $f(y)=0$ implies $g(y)=0$.
Since $g(y)$ is an increasing function on $(2, \infty)$, it has at most one root on $(2, \infty)$. Consequently, $f(y)$ has at most one root on $(2, \infty)$. Since $$\lim_{y \to 2} f(y) = (n+1) - (1+(4-x^2_0)m) n  = 1 - (4-x^2_0)mn>0$$ and $\lim_{y \to \infty} f(y)=-\infty$, we conclude that $f(y)$ has exactly one root on $(2, \infty)$.

Steps 1 and 2 complete the proof of Theorem \ref{main1}(i). 

\subsection{$J(2m,-2n)$} The knot $J(2m,-2n)$ is the two-bridge knot $\fb(4mn+1,4mn-2n+1)$. 

\subsubsection{The knot group and the Riley polynomial}

\begin{lemma} \label{w2}
Suppose $1 \le j \le 4mn$. Write $j=2mq+r$, where $0 \le r \le 2m-1$. Then 
$\ve_j = (-1)^{q+r-1}$ if $r \ge 1$, and $\ve_j = (-1)^q$ if $r=0$.
\end{lemma}

Lemma \ref{w2} implies that the Schubert presentation of the knot group of $J(2m,-2n)$ is $\pi_1(J(2m,-2n)) = \la a,b \mid wa=bw \ra$ where
$$w =  \big[ (ab^{-1})^{m}(a^{-1}b)^{m} \big]^{n}.$$

\begin{proposition}
One has $\Phi_{J(2m,-2n)}(x,y)= S_n(t)-\mu S_{n-1}(t)$ where 
\begin{eqnarray*}
t &=& 2+(y-2)(y+2-x^2)S^2_{m-1}(y),\\
\mu &=& 1 - (y+2-x^2)S_{m-1}(y)(S_{m-1}(y)-S_{m-2}(y)).
\end{eqnarray*}
\end{proposition}

\subsubsection{Proof of Theorem \ref{main1}(ii)} Suppose $x_0 \in \BR$ satisfies $\sqrt{4-(mn)^{-1}}< |x_0| \le 2$. We want to prove that the equation $\Phi_{J(2m,-2n)}(x_0,y)=0$ has no real solutions $y$. 

The proof consists of two steps.\\
\textit{Step 1: The equation $\Phi_{J(2m,-2n)}(x_0,y)=0$ has no real solutions $y \le 2$.}

The proof of Step 1 is similar to that of $J(2m,2n)$.\\
\textit{Step 2: The equation $\Phi_{J(2m,-2n)}(x_0,y)=0$ has no real solutions $y > 2$.}

Suppose $y>2$. Then $t>2$ and $\mu < 1$. Hence $$\Phi_{J(2m,-2n)}(x_0,y) = S_n(t) - \mu S_{n-1}(t) > S_{n}(t) - S_{n-1}(t) >0.$$
Steps 1 and 2 complete the proof of Theorem \ref{main1}(ii).

\section{Proof of Theorem \ref{main2}}

\subsection{$J(2m+1,2n)$} The knot $J(2m+1,2n)$ is the two-bridge knot $\fb(4mn+2n-1,4mn-1)$.

\subsubsection{The knot group and the Riley polynomial}

\begin{lemma} \label{w3}
Suppose $1 \le j \le 2n(2m+1)-2$. Write $j=(2m+1)q+r$, where $0 \le r \le 2m$. Then $\ve_j=(-1)^{r-1}$.
\end{lemma}

Lemma \ref{w3} implies that the Schubert presentation of the knot group of $J(2m+1,2n)$ is $\pi_1(J(2m+1,2n)) = \la a,b \mid wa=bw \ra$ where $$w=(ab^{-1})^m \big[ (a^{-1}b)^ma^{-1}b^{-1}(ab^{-1})^m \big]^{n-1}(a^{-1}b)^m.$$ 

\begin{proposition}
One has $\Phi_{J(2m+1,2n)}(x,y)= S_n(t) -  \mu S_{n-1}(t)$ where 
\begin{eqnarray*}
t &=& x^2-y-(y-2)(y+2-x^2)S_m(y)S_{m-1}(y),\\
\mu &=& 1 - (y-x^2+2)S_m(y) ( S_m(y)-S_{m-1}(y) ).
\end{eqnarray*}
\end{proposition}

\subsubsection{Proof of Theorem \ref{main2}(i)} Suppose $x_0 \in \BR$ satisfies $|x_0| \ge 2\cos \frac{\pi}{4m+2}$. We want to prove that the equation $\Phi_{J(2m+1,2n)}(x_0,y)=0$ has at least $n-1$ real solutions $y$. To achieve this, we will show that there exist real numbers $y_1 < y_2 < \cdots < y_n$ such that $f(y_j)f(y_{j+1})<0$ for $1 \le j \le n-1$, where $f(y)=\Phi_{J(2m+1,2n)}(x_0,y)$.

We have $S_n(t)-S_{n-1}(t)=\prod_{j=1}^n (t-2\cos \frac{(2j-1)\pi}{2n+1})$. Let $t_j=2\cos \frac{(2j-1)\pi}{2n+1}$. Since 
$$S_n(t_j)=\frac{\sin \frac{(2j-1)(n+1)\pi}{2n+1}} {\sin \frac{(2j-1)\pi}{2n+1}}=\frac{\sin \left( j\pi -\frac{\pi}{2}  + \frac{(2j-1)\pi}{2(2n+1)} \right)} {\sin \frac{(2j-1)\pi}{2n+1}}=(-1)^{j-1}\frac{\cos \frac{(2j-1)\pi}{2(2n+1)} } {\sin \frac{(2j-1)\pi}{2n+1}}$$
we have $(-1)^{j-1}S_n(t_j)>0$.

Note that $2-t=(y-x_0^2+2)(S_m(y)-S_{m-1}(y))^2.$ Choose the largest real solution $y_j$ of $2-t_j=(y-x_0^2+2)(S_m(y)-S_{m-1}(y))^2.$

\begin{lemma}
One has $x^2-2<y_1<y_2< \cdots < y_n$.
\end{lemma}

\begin{proof}
Let $g(y) = (y-x_0^2+2)(S_m(y)-S_{m-1}(y))^2$ where $y \in \BR$. We have $g(y_1)=2-t_1 < 2-t_2$. There exists $y>y_1$ such that $g(y)=2-t_2$. Hence $y_2 \ge y>y_1$. Other inequalities can be proved in a similar way.
\end{proof}

We have
$$f(y_j)=(1-\mu)S_n(t_j)=(y_j+2-x_0^2)S_m(y_j) ( S_{m}(y_j) - S_{m-1}(y_j) )S_n(t_j).$$
Since $y_j>x_0^2-2>2\cos \frac{\pi}{2m+1}$, we have $S_m(y_j)>S_{m-1}(y_j)>0.$ Moreover, $(-1)^{j-1}S_n(t_j)>0$. Hence $(-1)^{j-1} f(y_j)>0$. For each $1 \le j \le n-1$, we have $$f(y_j)f(y_{j+1})<0$$ and hence there exists $y'_j \in (y_j,y_{j+1})$ such that $f(y'_j)=0.$ There are at least $n-1$ real solutions $y$ of the equation $f(y)=0$.

\subsection{$J(2m+1,-2n)$} The knot $J(2m+1,-2n)$ is the two-bridge knot $\fb(4mn+2n+1,4mn+1)$.

\subsubsection{The knot group and the Riley polynomial}

\begin{lemma} \label{w4}
Suppose $1 \le j \le 2n(2m+1)$. Write $j=(2m+1)q+r$, where $0 \le r \le 2m$. Then $\ve_j=(-1)^{r-1}$ if $r \ge 1$, and $\ve_j=1$ if $r=0$.
\end{lemma}

Lemma \ref{w4} implies that the Schubert presentation of the knot group of $J(2m+1,-2n)$ is $\pi_1(J(2m+1,-2n)) = \la a,b \mid wa=bw \ra$ where
\begin{eqnarray*}
w &=&  \big[ (ab^{-1})^{m}ab(a^{-1}b)^{m} \big]^{n}.
\end{eqnarray*}

\begin{proposition}
One has $\Phi_{J(2m+1,-2n)}(x,y)= S_{n}(t) - \mu S_{n-1}(t)$ where 
\begin{eqnarray*}
t &=& x^2-y-(y-2)(y+2-x^2)S_m(y)S_{m-1}(y),\\
\mu &=& 1 + (y+2-x^2)S_{m-1}(y)(S_m(y)-S_{m-1}(y)).
\end{eqnarray*}
\end{proposition}

\subsubsection{Proof of Theorem \ref{main2}(i)} Suppose $x_0 \in \BR$ satisfies $|x_0| \ge 2\cos \frac{\pi}{4m+2}$. We want to prove that the equation $\Phi_{J(2m+1,-2n)}(x_0,y)=0$ has at least $n$ real solutions $y$. To achieve this, we will show that there exist real numbers $y_0 < y_1 < y_2 < \cdots < y_n$ such that $f(y_j)f(y_{j+1})<0$ for $0 \le j \le n-1$, where $f(y)=\Phi_{J(2m+1,-2n)}(x_0,y)$.

Choose $t_j$ and $y_j$ for $1 \le j \le n$ as in the proof of Theorem \ref{main2}(i). Then $x^2-2<y_1<y_2< \cdots < y_n$. We have 
$$f(y_j)=(1-\mu)S_n(t_j)=-(y+2-x_0^2)S_{m-1}(y)(S_m(y)-S_{m-1}(y))S_n(t_j),$$
which implies that $(-1)^{j}f(y_j)>0$.

Let $y_0=x_0^2-2$. Then $f(y_0)=S_n(2)-S_{n-1}(2)=1$. For each $0 \le j \le n-1$, we have $$f(y_j)f(y_{j+1})<0$$ and hence there exists $y'_j \in (y_j,y_{j+1})$ such that $f(y'_j)=0.$ There are at least $n$ real solutions $y$ of the equation $f(y)=0$.

\end{document}